\documentclass[]{interact}

\usepackage{epstopdf}
\usepackage[caption=false]{subfig}
\usepackage[numbers,sort&compress]{natbib}
\bibpunct{[}{]}{,}{n}{}{,}
\renewcommand\bibfont{\fontsize{10}{12}\selectfont}
\usepackage{mathrsfs}
\usepackage{amsmath,amsthm,amssymb}
\usepackage{mathtools}
\usepackage{hyperref}
\usepackage{cleveref}
\usepackage{booktabs}
\usepackage{array}
\usepackage{graphicx}
\usepackage{tikz}
\usepackage{float}
\usetikzlibrary{arrows.meta, positioning, fit, backgrounds}

\theoremstyle{plain}
\newtheorem{theorem}{Theorem}[section]
\newtheorem{lemma}[theorem]{Lemma}
\newtheorem{proposition}[theorem]{Proposition}
\newtheorem{corollary}[theorem]{Corollary}

\theoremstyle{definition}
\newtheorem{definition}[theorem]{Definition}
\newtheorem{example}[theorem]{Example}

\theoremstyle{remark}
\newtheorem{remark}[theorem]{Remark}
\newtheorem{observation}[theorem]{Observation}

\newcommand{\Crown}{\mathscr{C}}
\newcommand{\Skel}{\mathcal{L}}
\newcommand{\Tri}{\mathcal{T}}
\newcommand{\Znn}{\mathbb{Z}_{\geq 0}}
\newcommand{\Zpos}{\mathbb{Z}_{> 0}}

\begin{document}

\title{A Coordinate System for Collatz Dynamics}

\author{
\name{Jennifer Williams\thanks{Email: j.williams@soton.ac.uk, ORCID: \url{https://orcid.org/0000-0003-1410-0427}}}
\affil{School of Electronics and Computer Science\\ University of Southampton, United Kingdom}
}

\maketitle

\begin{abstract}
It is well-established that every odd positive integer $n$ can be written uniquely as $n = \lambda \cdot 2^a \cdot 3^b - 1$ where $\gcd(\lambda, 6) = 1$ and $a \geq 1$. Building from this 3-smooth factorization, we introduce a partition of the nonnegative integers into countably many infinite triangles where each row $k$ forms a \emph{Collatz chain} of alternating parity. The partition admits a coordinate system as a \emph{skeleton} $\Skel_\lambda$ using the pair $(a, b)$ for odd positive integers within a geometric structure where row $k$ corresponds to $k = a + b$. Each position $(a, b)$ maps to $(a-1, b+1)$, a deterministic diagonal flow requiring no number-theoretic input. At the boundary $a = 1$, the trajectory exits to another skeleton depending on the factorization of $\lambda \cdot 3^{b+1} - 1$. The coordinate system is new. As a concrete application, we prove that rows $k \equiv 2 \pmod 4$ with $k \geq 6$ in the principal skeleton $\Skel_1$ contain no primes, and show this is the \emph{unique} residue class admitting complete algebraic obstruction. Our contribution is the framework that makes visible which nonnegative integers these arguments apply to, with all results independent of the Collatz conjecture.
\end{abstract}


\section{Introduction}\label{sec:intro}

Despite almost a century of work on the Collatz problem, the geometric organization of its dynamics remains underexplored. Prior work on the problem is extensive, with key results spanning multiple areas such as density bounds~\cite{Applegate1995}, stochastic models~\cite{LagariasWeiss1992}, and cycle analysis~\cite{Knight2026}. A variety of geometric approaches to the Collatz problem have been developed. This paper introduces a coordinate system that differs from these approaches by organizing nonnegative integers via 3-smooth factorization.

The Collatz function
\[
C(n) = \begin{cases} n/2 & \text{if } n \equiv 0 \pmod 2, \\ 3n+1 & \text{if } n \equiv 1 \pmod 2 \end{cases}
\]
was introduced by Lothar Collatz in 1937 and is surveyed in~\cite{Lagarias2010}. We write $C$ for this full map throughout. Two related maps appear in the literature. The accelerated map $T(n) = (3n+1)/2$ for odd $n$ (and $T(n) = n/2$ for even $n$) removes a single factor of 2 after each odd step. The Syracuse function $S(n) = (3n+1)/2^{v_2(3n+1)}$ on odd $n$ removes all of them at once, mapping each odd number directly to the next. Our results concern $C$. The conjecture that iteration from any positive integer $n$ eventually reaches 1 remains open despite computational verification to $2^{68}$~\cite{Barina2021} and Tao's breakthrough~\cite{Tao2022} showing that almost all orbits attain almost bounded values, building on earlier density results of Terras \cite{terras1976stopping}.

The coordinate system introduced in this paper organizes Collatz dynamics geometrically as triangles and arises from computational exploration. Examining Collatz orbits, we identify maximal subsequences with strictly alternating parity. Unlike typical orbits, which may apply $n/2$ multiple times consecutively, these subsequences alternate parity at every step and cannot be extended further while preserving this alternation. We call them \emph{Collatz chains} and give a formal definition in Section~\ref{sec:partition}. We show that the specific arrangement of such chains illuminates a geometric structure that aligns exactly to known prior work on the Collatz problem. 

Computational enumeration for $n \leq 10^7$ reveals that every natural number belongs to exactly one Collatz chain, and the distribution of chain lengths follows a geometric law (see Figure~\ref{fig:chain-lengths}). Investigating this regularity leads us to discover that the chains organize into countably many infinite triangles. We call these \emph{crown triangles} because each is seeded by a single element at row $k=0$, which we call the \emph{crown}. The odd part of each crown triangle forms an \emph{odd skeleton}. The skeleton structure admits a closed-form description involving 3-smooth numbers, and in turn yields proofs of nontrivial theorems, thereby connecting this geometric arrangement to prior work on the Collatz problem.

\begin{figure}[ht]
\centering
\includegraphics[width=\textwidth]{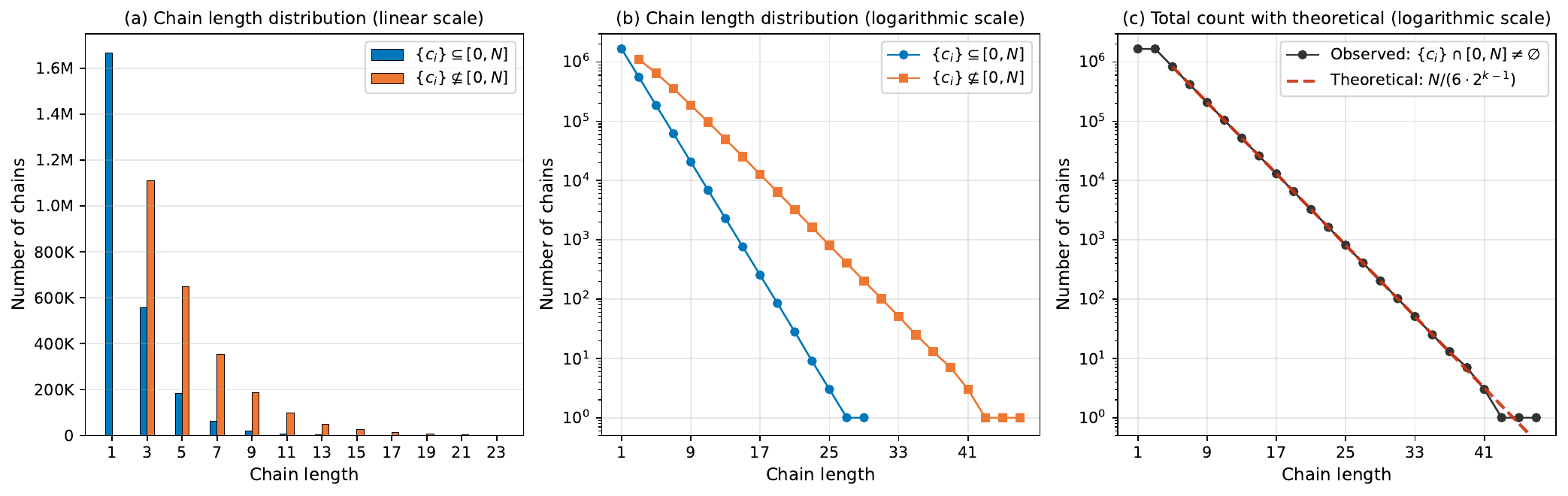}
\caption{Distribution of Collatz chain lengths for $n \in \{0, 1, \ldots, N\}$ with $N = 10^7$. Chains are categorized as $\{c_i\} \subseteq [0, N]$ (all elements $\leq N$) or $\{c_i\} \not\subseteq [0, N]$ (at least one element $> N$). (a) Linear scale showing chains of length 1, 3, 5, \ldots, 23. (b) Logarithmic scale showing all observed chain lengths from 1 to 47. (c) Total count of chains intersecting $[0, N]$ compared with theoretical prediction $N/(6 \cdot 2^{k-1})$.}
\label{fig:chain-lengths}
\end{figure}

The key idea is a coordinate system for nonnegative odd integers. 3-smooth factorization admits that every odd positive integer $n$ can be written uniquely as
\[
n = \lambda \cdot 2^a \cdot 3^b - 1
\]
where $\gcd(\lambda, 6) = 1$, $a \geq 1$, and $b \geq 0$, which follows from writing $n + 1 = 2^a \cdot 3^b \cdot \lambda$ and extracting the powers of 2 and 3. This elementary factorization links to our geometric structures (crown triangles, odd skeletons) as coordinates. The parameter $\lambda$ indexes which odd skeleton a nonnegative integer belongs to. The pair $(a, b)$ gives coordinates within the odd skeleton triangle, and the sum $k = a + b$ gives the row.

\subsection{Summary of Results}

Our main contributions are the following.

The nonnegative integers partition into crown triangles $\{\Tri_c : c \in \Crown\}$, indexed by crowns $c \equiv 0$ or $8 \pmod{12}$. Each row of each triangle is a Collatz chain. This is Theorem~\ref{thm:partition}.

The odd elements of each triangle form an odd skeleton $\Skel_\lambda$ consisting of integers $\lambda \cdot 2^a \cdot 3^b - 1$ with $a \geq 1$ and $b \geq 0$. The parameter $\lambda = (c+2)/2$ ranges over positive integers coprime to 6. This is Proposition~\ref{prop:skeleton}.

Within each skeleton, the Collatz map sends position $(a, b)$ to $(a-1, b+1)$ until reaching the boundary $a = 1$. The row index $k = a + b$ is preserved. This diagonal flow is completely deterministic. This is Theorem~\ref{thm:diagonal}.

The interior of each skeleton, where $a \geq 2$, requires no number-theoretic input. The 2-adic valuation satisfies $v_2(3n+1) = 1$ always. At the boundary $a = 1$, the 2-adic valuation depends on the factorization of $\lambda \cdot 3^{b+1} - 1$, and the trajectory exits to another skeleton. The number-theoretic difficulty of Collatz is concentrated at these boundary transitions.

In the principal skeleton $\Skel_1$, rows $k \equiv 2 \pmod 4$ with $k \geq 6$ contain no primes. The proof uses divisibility by 5 and difference-of-squares factorization. This is Theorem~\ref{thm:zero-prime}. The Complete Covering Theorem (Theorem~\ref{thm:complete-covering}) explains why this residue class is special. It is the unique class modulo 4 where every position admits an algebraic obstruction to primality.

\subsection{What Is New}
The organization of nonnegative odd integers by the pair $(a, b)$ from the factorization $n + 1 = \lambda \cdot 2^a \cdot 3^b$, together with the arrangement by weight $k = a + b$ into rows whose elements are Collatz chains, appears to be new. The individual ingredients are not new. The factorization itself is elementary. There is substantial prior work on 3-smooth numbers \cite{Blecksmith1998}. And the family $\{2^a \cdot 3^b - 1\}$ has been studied as the Pierpont numbers of the second kind \cite{Guy2004}. Our contribution is specifically the synthesis of the coordinate system that links this factorization to the Collatz dynamics, ultimately making certain patterns more visible.

The divisibility arguments underlying Theorem~\ref{thm:zero-prime} are also elementary. The Complete Covering Theorem shows that these arguments completely characterize the algebraically prime-free rows in $\Skel_1$. The result is stated and proved in coordinates that did not previously exist.

The partition of nonnegative integers presented in Theorem~\ref{thm:partition} is independent of the Collatz conjecture. This structural result holds whether or not every Collatz orbit reaches 1. 

\subsection{Relationship to Prior Work}
Several strands of prior work relate to the framework presented here.

\subsubsection{Geometric and structural approaches to Collatz}
The Collatz problem admits several geometric perspectives. The Collatz graph, where nonnegative integers are nodes and $C$ defines directed edges, has been studied extensively. Wirsching's monograph~\cite{Wirsching1998} develops its tree structure in detail, and related tree-based analyses underpin the density bounds of Applegate and Lagarias~\cite{Applegate1995}. A distinct line of work organizes integers by residue class modulo small numbers to track Collatz dynamics. Kannan and Moorthy~\cite{Kannan2016} analyze orbits modulo an integer, and Ren~\cite{Ren2023} gives a regular partition by residue class. Knight~\cite{Knight2026} uses Christoffel words to characterize the parity vector of a high cycle, giving a combinatorial-geometric criterion for cycle existence. The crown triangle framework differs from these approaches in two respects. First, it organizes integers by the 3-smooth factorization of $n + 1$ rather than by residues of $n$ itself, so that the parameter $k = a + b$ corresponds directly to the length of a Collatz chain. Second, the skeleton coordinates $(a, b)$ render the deterministic portion of Collatz dynamics as diagonal lattice translation, isolating the number-theoretic content at the boundary $a = 1$. The relationship between Knight's parity vector framework and paths in the skeleton lattice appears especially worth pursuing. Odd steps of the Collatz function correspond to diagonal moves $(a, b) \mapsto (a - 1, b + 1)$, and boundary exits correspond to transitions between skeletons. Altogether, the coordinate system admits a 3-dimensional geometric structure ($\lambda, a, b$) that to our knowledge has not been examined in this form.

\subsubsection{3-smooth numbers and the Pierpont family}
The skeleton $\Skel_1$ enumerates nonnegative integers one less than a 3-smooth number. Primes in this family are the Pierpont primes of the second kind~\cite{Guy2004}, and include Mersenne primes (column $b = 0$) and Thabit primes (column $b = 1$) as classical subfamilies. Blecksmith et al.~\cite{Blecksmith1998} studied the combinatorics of 3-smooth representations of integers. The crown framework organizes the Pierpont family by weight $a + b$, which is the natural invariant for Collatz chain length but does not appear to have been used as an organizing principle for this family previously. Prime distribution questions for thin sequences of this form are governed by the Bateman-Horn heuristic~\cite{BatemanHorn1962}, which predicts a positive density of primes along generic arithmetic progressions within the Pierpont family. Theorem~\ref{thm:zero-prime} identifies the residue class where this heuristic is obstructed by elementary divisibility.

\subsubsection{Growth rates and $(3/2)$-dynamics}
Through our coordinate system, a diagonal flow can be found. The diagonal flow $(a, b) \mapsto (a - 1, b + 1)$ multiplies $n + 1$ by a factor of $3/2$ at each step. This factor connects the interior skeleton dynamics notationally to the $\beta$-transformation $x \mapsto (3/2)x \bmod 1$ studied by Flatto~\cite{Flatto1992} and to Mahler's work on powers of $3/2$~\cite{Mahler1968}. We note the correspondence here as a structural feature rather than a technical tool. Whether the skeleton framework yields genuinely new leverage on $(3/2)$-dynamics is an open question.

\subsubsection{The Syracuse function}
Recall the Syracuse function $S(n) = (3n + 1) / 2^{v_2(3n+1)}$~\cite{Applegate1995}. The 2-adic valuation $v_2(3n+1)$ is the Syracuse coefficient, and it controls the multiplicative factor by which each odd-to-odd step changes the value of $n$. Corollary~\ref{cor:valuation} shows that interior skeleton positions within our coordinate system (those with $a \geq 2$) achieve $v_2(3n+1) = 1$, the minimum possible value. The skeleton interior is therefore precisely where Syracuse dynamics are as contractive as possible, and the boundary $a = 1$ is where higher Syracuse coefficients occur.

\section{The Crown Triangle Partition}\label{sec:partition}

We present the crown triangle structure algebraically, deriving modular properties as consequences. The mod 12 structure of crowns connects to prior work on Collatz behavior in residue classes \cite{Kannan2016,Ren2023}.

\begin{definition}[Collatz chain]\label{def:collatz-chain}
A \emph{Collatz chain} is a finite sequence $(n_0, n_1, \ldots, n_m)$ of nonnegative integers satisfying $C(n_i) = n_{i+1}$ for $0 \leq i < m$, with $n_i$ and $n_{i+1}$ of opposite parity for all such $i$, and which is maximal in the sense that extending the sequence by one step in either direction would violate the parity alternation property.
\end{definition}

\subsection{Crowns and the $\lambda$-Parameter}

\begin{definition}[Crown]
A nonnegative integer $c$ is a \emph{crown} if $c \equiv 0$ or $8 \pmod{12}$. The set of crowns is
\[
\Crown = \{0, 8, 12, 20, 24, 32, 36, 44, 48, \ldots\}.
\]
For each crown $c$, define the $\lambda$-parameter by $\lambda(c) = (c+2)/2$.
\end{definition}

\begin{proposition}[$\lambda$-bijection]\label{prop:lambda-bijection}
The map $c \mapsto \lambda(c)$ is a bijection between crowns and positive integers coprime to 6:
\[
\{\lambda(c) : c \in \Crown\} = \{n \in \Zpos : \gcd(n, 6) = 1\} = \{1, 5, 7, 11, 13, 17, 19, 23, \ldots\}.
\]
\end{proposition}

\begin{proof}
If $c \equiv 0 \pmod{12}$, then $c + 2 \equiv 2 \pmod{12}$, so $\lambda = (c+2)/2 \equiv 1 \pmod{6}$. If $c \equiv 8 \pmod{12}$, then $c + 2 \equiv 10 \pmod{12}$, so $\lambda = (c+2)/2 \equiv 5 \pmod{6}$. Every positive integer coprime to 6 is congruent to 1 or 5 modulo 6. The map is injective since $c = 2\lambda - 2$ is determined by $\lambda$.
\end{proof}

The first few correspondences are: $c = 0 \leftrightarrow \lambda = 1$, $c = 8 \leftrightarrow \lambda = 5$, $c = 12 \leftrightarrow \lambda = 7$, $c = 20 \leftrightarrow \lambda = 11$.

\subsection{The Residual Triangle}

The residual triangle is a template of 3-smooth numbers that underlies all crown triangles.

\begin{definition}[Residual triangle]
For integers $k \geq 0$ and $0 \leq j \leq 2k$, define
\[
R(k, j) = 2^{k - \lfloor j/2 \rfloor} \cdot 3^{\lfloor j/2 \rfloor}.
\]
Row $k$ of $R$ contains $2k+1$ elements at positions $j = 0, 1, \ldots, 2k$.
\end{definition}

The elements of $R$ are 3-smooth numbers organized by weight, where the weight of $2^a \cdot 3^b$ is $a + b$. Row $k$ contains elements of weight $k$.

\begin{example}\label{ex:residual}
The first rows of $R$ are:
\begin{center}
\begin{tabular}{c|ccccccccccc}
$k$ & \multicolumn{11}{c}{$R(k,j)$ for $j = 0, 1, \ldots, 2k$} \\
\midrule
0 & 1 \\
1 & 2 & 2 & 3 \\
2 & 4 & 4 & 6 & 6 & 9 \\
3 & 8 & 8 & 12 & 12 & 18 & 18 & 27 \\
4 & 16 & 16 & 24 & 24 & 36 & 36 & 54 & 54 & 81
\end{tabular}
\end{center}
Consecutive positions $j$ and $j+1$ share the same $R$ value when $j$ is even.
\end{example}

\begin{lemma}[Parity structure of $R$]\label{lem:R-parity}
For all valid $j$:
\begin{enumerate}
\item If $j$ is even, then $R(k, j) = 2^{k - j/2} \cdot 3^{j/2}$.
\item If $j$ is odd, then $R(k, j) = R(k, j-1)$.
\item $R(k, j+1) = \frac{3}{2} R(k, j)$ when $j$ is odd.
\end{enumerate}
\end{lemma}

\begin{proof}
For even $j$, we have $\lfloor j/2 \rfloor = j/2$, giving $R(k,j) = 2^{k-j/2} \cdot 3^{j/2}$.

For odd $j = 2m+1$, we have $\lfloor j/2 \rfloor = m$, so $R(k,j) = 2^{k-m} \cdot 3^m$. At $j-1 = 2m$, we have $\lfloor (j-1)/2 \rfloor = m$, giving $R(k, j-1) = 2^{k-m} \cdot 3^m = R(k,j)$.

For odd $j$, the next position $j+1$ is even, so $R(k, j+1) = 2^{k - (j+1)/2} \cdot 3^{(j+1)/2}$. Since $R(k,j) = 2^{k - (j-1)/2} \cdot 3^{(j-1)/2}$, we have $R(k, j+1)/R(k, j) = 2^{-1} \cdot 3 = 3/2$.
\end{proof}

\subsection{Crown Triangles}

\begin{definition}[Crown triangle]\label{def:crown-triangle}
For a crown $c$ with $\lambda = \lambda(c)$, define the parity multiplier $g(j) = 2$ if $j$ is even and $g(j) = 1$ if $j$ is odd. The crown triangle $\Tri_c$ has element at position $(k, j)$ given by
\[
\Tri_c(k, j) = g(j) \cdot (\lambda \cdot R(k, j) - 1).
\]
Row $k$ contains positions $j = 0, 1, \ldots, 2k$.
\end{definition}

The parity multiplier ensures that even positions contain even integers and odd positions contain odd integers.

\begin{figure}[ht]
\centering
\begin{tikzpicture}[scale=0.55, every node/.style={font=\footnotesize}]
  
  \begin{scope}[xshift=-7cm]
    \node at (0, 5.2) {\textbf{$\Tri_0$ (crown $\equiv 0 \mod 12$)}};
    
    \node[circle, draw, dashed, thick, fill=gray!10, minimum size=6mm, inner sep=1pt] (A-c0) at (0, 3.5) {0};
    
    \node[circle, draw, fill=blue!20, minimum size=6mm, inner sep=1pt] (A-l1) at (-1.5, 2) {4};
    \node[circle, draw, minimum size=6mm, inner sep=1pt] (A-m1) at (0, 2) {\textbf{1}};
    \node[circle, draw, fill=orange!30, minimum size=6mm, inner sep=1pt] (A-r1) at (1.5, 2) {2};
    
    \node[circle, draw, fill=blue!20, minimum size=6mm, inner sep=1pt] (A-l2) at (-3, 0.5) {16};
    \node[circle, draw, minimum size=6mm, inner sep=1pt] (A-m2a) at (-1.5, 0.5) {\textbf{5}};
    \node[circle, draw, minimum size=6mm, inner sep=1pt] (A-m2b) at (0, 0.5) {10};
    \node[circle, draw, minimum size=6mm, inner sep=1pt] (A-m2c) at (1.5, 0.5) {\textbf{3}};
    \node[circle, draw, fill=orange!30, minimum size=6mm, inner sep=1pt] (A-r2) at (3, 0.5) {6};
    
    \node[circle, draw, fill=blue!20, minimum size=6mm, inner sep=1pt] (A-l3) at (-4.5, -1) {52};
    \node[circle, draw, minimum size=6mm, inner sep=1pt] (A-m3a) at (-3, -1) {\textbf{17}};
    \node[circle, draw, minimum size=6mm, inner sep=1pt] (A-m3b) at (-1.5, -1) {34};
    \node[circle, draw, minimum size=6mm, inner sep=1pt] (A-m3c) at (0, -1) {\textbf{11}};
    \node[circle, draw, minimum size=6mm, inner sep=1pt] (A-m3d) at (1.5, -1) {22};
    \node[circle, draw, minimum size=6mm, inner sep=1pt] (A-m3e) at (3, -1) {\textbf{7}};
    \node[circle, draw, fill=orange!30, minimum size=6mm, inner sep=1pt] (A-r3) at (4.5, -1) {14};
    
    \draw[thick, blue!60] (A-c0) -- (A-l1) -- (A-l2) -- (A-l3);
    \draw[thick, orange!70] (A-c0) -- (A-r1) -- (A-r2) -- (A-r3);
    \draw[thick, blue!60, ->, dashed]   (A-l3) -- ++(-1.2,-1.2);
    \draw[thick, orange!70, ->, dashed] (A-r3) -- ++(1.2,-1.2);
    \draw[->, gray, semithick] (A-r1) -- (A-m1);
    \draw[->, gray, semithick] (A-m1) -- (A-l1);
    
    \draw[->, gray, semithick] (A-r2) -- (A-m2c);
    \draw[->, gray, semithick] (A-m2c) -- (A-m2b);
    \draw[->, gray, semithick] (A-m2b) -- (A-m2a);
    \draw[->, gray, semithick] (A-m2a) -- (A-l2);
    
    \draw[->, gray, semithick] (A-r3) -- (A-m3e);
    \draw[->, gray, semithick] (A-m3e) -- (A-m3d);
    \draw[->, gray, semithick] (A-m3d) -- (A-m3c);
    \draw[->, gray, semithick] (A-m3c) -- (A-m3b);
    \draw[->, gray, semithick] (A-m3b) -- (A-m3a);
    \draw[->, gray, semithick] (A-m3a) -- (A-l3);
  \end{scope}
  
  \begin{scope}[xshift=7cm]
    \node at (0, 5.2) {\textbf{$\Tri_8$ (crown $\equiv 8 \mod 12$)}};
    
    \node[circle, draw, dashed, thick, fill=gray!10, minimum size=6mm, inner sep=1pt] (B-c0) at (0, 3.5) {8};
    
    \node[circle, draw, fill=blue!20, minimum size=6mm, inner sep=1pt] (B-l1) at (-1.5, 2) {28};
    \node[circle, draw, minimum size=6mm, inner sep=1pt] (B-m1) at (0, 2) {\textbf{9}};
    \node[circle, draw, fill=orange!30, minimum size=6mm, inner sep=1pt] (B-r1) at (1.5, 2) {18};
    
    \node[circle, draw, fill=blue!20, minimum size=6mm, inner sep=1pt] (B-l2) at (-3, 0.5) {88};
    \node[circle, draw, minimum size=6mm, inner sep=1pt] (B-m2a) at (-1.5, 0.5) {\textbf{29}};
    \node[circle, draw, minimum size=6mm, inner sep=1pt] (B-m2b) at (0, 0.5) {58};
    \node[circle, draw, minimum size=6mm, inner sep=1pt] (B-m2c) at (1.5, 0.5) {\textbf{19}};
    \node[circle, draw, fill=orange!30, minimum size=6mm, inner sep=1pt] (B-r2) at (3, 0.5) {38};
    
    \node[circle, draw, fill=blue!20, minimum size=6mm, inner sep=1pt] (B-l3) at (-4.5, -1) {268};
    \node[circle, draw, minimum size=6mm, inner sep=1pt] (B-m3a) at (-3, -1) {\textbf{89}};
    \node[circle, draw, minimum size=6mm, inner sep=1pt] (B-m3b) at (-1.5, -1) {178};
    \node[circle, draw, minimum size=6mm, inner sep=1pt] (B-m3c) at (0, -1) {\textbf{59}};
    \node[circle, draw, minimum size=6mm, inner sep=1pt] (B-m3d) at (1.5, -1) {118};
    \node[circle, draw, minimum size=6mm, inner sep=1pt] (B-m3e) at (3, -1) {\textbf{39}};
    \node[circle, draw, fill=orange!30, minimum size=6mm, inner sep=1pt] (B-r3) at (4.5, -1) {78};
    
    \draw[thick, blue!60] (B-c0) -- (B-l1) -- (B-l2) -- (B-l3);
    \draw[thick, orange!70] (B-c0) -- (B-r1) -- (B-r2) -- (B-r3);
    \draw[thick, blue!60, ->, dashed]   (B-l3) -- ++(-1.2,-1.2);
    \draw[thick, orange!70, ->, dashed] (B-r3) -- ++(1.2,-1.2);
    \draw[->, gray, semithick] (B-r1) -- (B-m1);
    \draw[->, gray, semithick] (B-m1) -- (B-l1);
    
    \draw[->, gray, semithick] (B-r2) -- (B-m2c);
    \draw[->, gray, semithick] (B-m2c) -- (B-m2b);
    \draw[->, gray, semithick] (B-m2b) -- (B-m2a);
    \draw[->, gray, semithick] (B-m2a) -- (B-l2);
    
    \draw[->, gray, semithick] (B-r3) -- (B-m3e);
    \draw[->, gray, semithick] (B-m3e) -- (B-m3d);
    \draw[->, gray, semithick] (B-m3d) -- (B-m3c);
    \draw[->, gray, semithick] (B-m3c) -- (B-m3b);
    \draw[->, gray, semithick] (B-m3b) -- (B-m3a);
    \draw[->, gray, semithick] (B-m3a) -- (B-l3);
  \end{scope}
  
\end{tikzpicture}
\caption{Crown triangles from the two crown families. Left: $\Tri_0$ with crown $\equiv 0 \mod 12$. Right: $\Tri_8$ with crown $\equiv 8 \mod 12$. Crowns are marked with dashed borders and grey fill. Orange edges (right side of each triangle) show recurrence $2n+2$; Blue edges (left side) show recurrence $3n+4$. Odd elements are shown in bold. Gray arrows indicate Collatz function flow along each row. Triangles are infinite and extend indefinitely.}
\label{fig:triangles}
\end{figure}
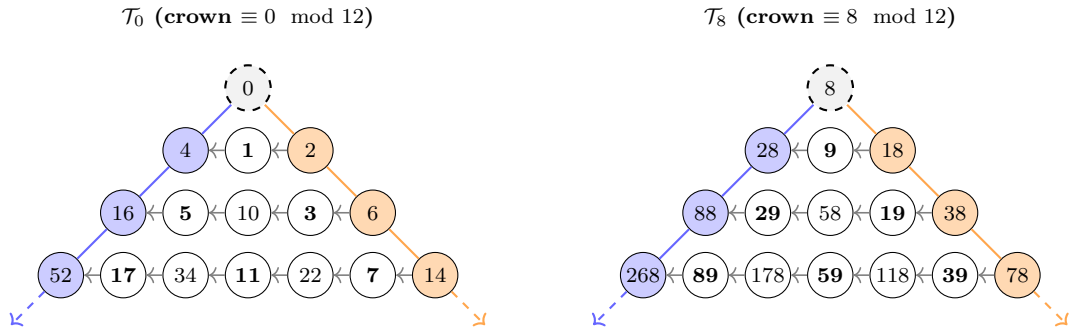

\begin{remark}
The apex element is the crown itself. At position $(0, 0)$ we have $R(0, 0) = 1$ and $g(0) = 2$, so $\Tri_c(0, 0) = 2(\lambda - 1) = 2\lambda - 2 = c$.
\end{remark}

\begin{example}\label{ex:triangles}
For crown $c = 0$ with $\lambda = 1$:
\begin{center}
\begin{tabular}{c|ccccccccccc}
$k$ & \multicolumn{11}{c}{$\Tri_0(k,j)$} \\
\midrule
0 & 0 \\
1 & 2 & 1 & 4 \\
2 & 6 & 3 & 10 & 5 & 16 \\
3 & 14 & 7 & 22 & 11 & 34 & 17 & 52 \\
4 & 30 & 15 & 46 & 23 & 70 & 35 & 106 & 53 & 160
\end{tabular}
\end{center}
For crown $c = 8$ with $\lambda = 5$:
\begin{center}
\begin{tabular}{c|ccccccccccc}
$k$ & \multicolumn{11}{c}{$\Tri_8(k,j)$} \\
\midrule
0 & 8 \\
1 & 18 & 9 & 28 \\
2 & 38 & 19 & 58 & 29 & 88 \\
3 & 78 & 39 & 118 & 59 & 178 & 89 & 268
\end{tabular}
\end{center}
\end{example}

\begin{proposition}[Parity of elements]\label{prop:parity}
$\Tri_c(k, j)$ is even if and only if $j$ is even.
\end{proposition}

\begin{proof}
For even $j$, we have $g(j) = 2$, so $\Tri_c(k, j) = 2(\lambda R(k, j) - 1)$, which is even.

For odd $j$, we have $g(j) = 1$, so $\Tri_c(k, j) = \lambda R(k, j) - 1$. Write $j = 2m + 1$ where $0 \leq m \leq k - 1$. Then $R(k, j) = 2^{k-m} \cdot 3^m$. Since $m \leq k - 1$, we have $k - m \geq 1$, so $R(k, j)$ is even. Thus $\lambda R(k, j)$ is even, and $\lambda R(k, j) - 1$ is odd.
\end{proof}

\subsection{The Row-Walking Theorem}

The Collatz function traverses each row sequentially.

\begin{theorem}[Row-Walking]\label{thm:row-walking}
For any crown $c$ and $k \geq 1$, the Collatz function satisfies
\[
C(\Tri_c(k, j)) = \Tri_c(k, j+1) \quad \text{for } 0 \leq j < 2k.
\]
\end{theorem}

\begin{proof}
We handle two cases based on the parity of $j$.

\textbf{Case 1: $j$ even.} Then $\Tri_c(k, j) = 2(\lambda R(k, j) - 1)$ is even. Applying $C$:
\[
C(\Tri_c(k, j)) = \frac{\Tri_c(k, j)}{2} = \lambda R(k, j) - 1.
\]
Since $j+1$ is odd, we have $g(j+1) = 1$. By Lemma~\ref{lem:R-parity}, $R(k, j+1) = R(k, j)$ when $j$ is even. Thus
\[
\Tri_c(k, j+1) = \lambda R(k, j+1) - 1 = \lambda R(k, j) - 1 = C(\Tri_c(k, j)).
\]

\textbf{Case 2: $j$ odd.} Then $\Tri_c(k, j) = \lambda R(k, j) - 1$ is odd. Applying $C$:
\[
C(\Tri_c(k, j)) = 3(\lambda R(k, j) - 1) + 1 = 3\lambda R(k, j) - 2.
\]
Since $j+1$ is even, we have $g(j+1) = 2$. By Lemma~\ref{lem:R-parity}, $R(k, j+1) = \frac{3}{2}R(k, j)$ when $j$ is odd. Thus
\[
\Tri_c(k, j+1) = 2\left(\lambda \cdot \frac{3}{2}R(k, j) - 1\right) = 3\lambda R(k, j) - 2 = C(\Tri_c(k, j)). \qedhere
\]
\end{proof}

\begin{corollary}[Collatz chains]\label{cor:chains}
Row $k$ of $\Tri_c$ forms a Collatz chain of length $2k+1$:
\[
\Tri_c(k, 0) \to \Tri_c(k, 1) \to \cdots \to \Tri_c(k, 2k).
\]
\end{corollary}

\subsection{The Partition Theorem}

\begin{theorem}[Partition]\label{thm:partition}
The crown triangles partition the nonnegative integers:
\[
\Znn = \bigsqcup_{c \in \Crown} \Tri_c.
\]
Every nonnegative integer belongs to exactly one crown triangle at exactly one position.
\end{theorem}

We prove this by constructing an explicit inverse to the map $(c, k, j) \mapsto \Tri_c(k, j)$.

\begin{lemma}[Inverse map for odd integers]\label{lem:inverse-odd}
Let $n$ be an odd positive integer. Write $n + 1 = 2^a \cdot 3^b \cdot \lambda$ where $a \geq 1$, $b \geq 0$, and $\gcd(\lambda, 6) = 1$. This factorization is unique. Then $n = \Tri_c(k, j)$ where $c = 2\lambda - 2$, $k = a + b$, and $j = 2b + 1$.
\end{lemma}

\begin{proof}
Since $n$ is odd, $n + 1$ is even, so $a \geq 1$. The factorization is unique by the fundamental theorem of arithmetic.

We verify that $n = \Tri_c(k, j)$. Since $j = 2b + 1$ is odd, we have $g(j) = 1$ and
\[
\Tri_c(k, j) = \lambda \cdot R(k, j) - 1.
\]
At odd position $j = 2b + 1$, we have $\lfloor j/2 \rfloor = b$, so
\[
R(k, j) = 2^{k - b} \cdot 3^b = 2^{(a+b) - b} \cdot 3^b = 2^a \cdot 3^b.
\]
Thus $\Tri_c(k, j) = \lambda \cdot 2^a \cdot 3^b - 1 = (n + 1) - 1 = n$.

\end{proof}

\begin{lemma}[Inverse map for even integers]\label{lem:inverse-even}
Let $n$ be a positive even integer. Write $n/2 + 1 = 2^{a'} \cdot 3^{b'} \cdot \lambda$ where $a', b' \geq 0$ and $\gcd(\lambda, 6) = 1$. Then $n = \Tri_c(k, j)$ where $c = 2\lambda - 2$, $k = a' + b'$, and $j = 2b'$.
\end{lemma}

\begin{proof}
Since $j = 2b'$ is even, we have $g(j) = 2$ and
\[
\Tri_c(k, j) = 2(\lambda \cdot R(k, j) - 1).
\]
At even position $j = 2b'$, we have $R(k, j) = 2^{a'} \cdot 3^{b'}$. Thus
\[
\Tri_c(k, j) = 2(\lambda \cdot 2^{a'} \cdot 3^{b'} - 1) = 2(n/2 + 1 - 1) = n.
\]
The constraint $0 \leq j \leq 2k$ requires $0 \leq 2b' \leq 2(a' + b')$, which holds since $a', b' \geq 0$.
\end{proof}

\begin{proof}[Proof of Theorem~\ref{thm:partition}]
For $n = 0$: We have $n/2 + 1 = 1 = 2^0 \cdot 3^0 \cdot 1$, giving $\lambda = 1$, $a' = b' = 0$, $k = 0$, $j = 0$, and $c = 0$. Indeed, $\Tri_0(0, 0) = 2(1 \cdot 1 - 1) = 0$.

For $n > 0$: Lemmas~\ref{lem:inverse-odd} and~\ref{lem:inverse-even} establish that every positive integer appears in exactly one crown triangle. The inverse maps are well-defined because prime factorization is unique.

Injectivity follows from the explicit inverse: given $n$, the parameters $(c, k, j)$ are uniquely determined.
\end{proof}

\begin{lemma}[Edge formulas]\label{lem:edges}
Two edges (position $j = 0$ and position $j = 2k$) of row $k$ satisfy:
\begin{align*}
\Tri_c(k, 0) &= 2(2^k \lambda - 1), \\
\Tri_c(k, 2k) &= 2(3^k \lambda - 1).
\end{align*}
\end{lemma}

\begin{proof}
At $j = 0$: $R(k, 0) = 2^k$, $g(0) = 2$, so $\Tri_c(k, 0) = 2(\lambda \cdot 2^k - 1)$.

At $j = 2k$: $R(k, 2k) = 3^k$, $g(2k) = 2$, so $\Tri_c(k, 2k) = 2(\lambda \cdot 3^k - 1)$.
\end{proof}

\section{The Odd Skeleton and Diagonal Dynamics}\label{sec:skeleton}

The odd elements of each crown triangle form a structured subset called the odd skeleton. The skeleton admits coordinates that make Collatz dynamics transparent.

\begin{definition}[Odd skeleton]
The odd skeleton $\Skel_\lambda$ consists of the odd elements of $\Tri_c$ where $\lambda = (c+2)/2$. We index these elements by coordinates $(a, b)$ where $a$ is the exponent of 2 and $b$ is the exponent of 3 in the factorization $n + 1 = \lambda \cdot 2^a \cdot 3^b$.
\end{definition}

\begin{remark}
For visualization in Section~\ref{sec:observations}, we extend $\lambda$ to all integers in a Collatz orbit by assigning each integer the $\lambda$ value of the next crown triangle reached along its forward Collatz orbit, but this extension is not used in any proof. 
\end{remark}

\begin{proposition}[Skeleton structure]\label{prop:skeleton}
The odd skeleton $\Skel_\lambda$ consists of exactly the integers
\[
\Skel_\lambda = \{\lambda \cdot 2^a \cdot 3^b - 1 : a \geq 1, b \geq 0\}.
\]
The element at position $(a, b)$ lies in row $k = a + b$ of $\Tri_c$.
\end{proposition}

\begin{proof}
By Lemma~\ref{lem:inverse-odd}, every nonnegative odd integer $n$ with $n + 1 = \lambda \cdot 2^a \cdot 3^b$ appears at position $j = 2b + 1$ in row $k = a + b$. The odd positions in row $k$ are $j = 1, 3, 5, \ldots, 2k-1$, corresponding to $b = 0, 1, \ldots, k-1$. For each such $b$, we have $a = k - b \geq 1$.
\end{proof}


\begin{example}
The principal skeleton $\Skel_1$ consists of nonnegative integers $2^a \cdot 3^b - 1$:
\begin{center}
\begin{tabular}{c|ccccc}
$k = a + b$ & $b = 0$ & $b = 1$ & $b = 2$ & $b = 3$ & $b = 4$ \\
\midrule
1 & 1 \\
2 & 3 & 5 \\
3 & 7 & 11 & 17 \\
4 & 15 & 23 & 35 & 53 \\
5 & 31 & 47 & 71 & 107 & 161
\end{tabular}
\end{center}
The column $b = 0$ contains the Mersenne numbers $2^k - 1$. The column $b = 1$ contains Thabit numbers $2^{k-1} \cdot 3 - 1$.
\end{example}

\begin{remark}
The skeleton coordinates $(a, b)$ and the crown triangle coordinates $(k, j)$ are related by $k = a + b$ and $j = 2b + 1$. Row $k$ of the odd skeleton $\Skel_\lambda$ consists of the $k$ odd positions of row $k$ in $\Tri_c$, corresponding to $b = 0, 1, \ldots, k - 1$ with $a = k - b$.
\end{remark}

\subsection{Diagonal Dynamics}

The Collatz map has a simple form when restricted to the skeleton interior.

\begin{theorem}[Diagonal Dynamics]\label{thm:diagonal}
Let $n = \lambda \cdot 2^a \cdot 3^b - 1$ with $a \geq 2$. Under the Collatz map, the next odd number in the trajectory is $\lambda \cdot 2^{a-1} \cdot 3^{b+1} - 1$. In skeleton coordinates, position $(a, b)$ maps to $(a-1, b+1)$.
\end{theorem}

\begin{proof}
Compute:
\[
3n + 1 = 3(\lambda \cdot 2^a \cdot 3^b - 1) + 1 = \lambda \cdot 2^a \cdot 3^{b+1} - 2 = 2(\lambda \cdot 2^{a-1} \cdot 3^{b+1} - 1).
\]
Since $a \geq 2$, we have $a - 1 \geq 1$, so $\lambda \cdot 2^{a-1} \cdot 3^{b+1}$ is even. Thus $\lambda \cdot 2^{a-1} \cdot 3^{b+1} - 1$ is odd, and $3n + 1 = 2 \times (\text{odd})$. Exactly one division by 2 yields the next odd number.
\end{proof}

\begin{corollary}[Interior 2-adic valuation]\label{cor:valuation}
For $n = \lambda \cdot 2^a \cdot 3^b - 1$ with $a \geq 2$, we have $v_2(3n+1) = 1$.
\end{corollary}

This connects to the work of Applegate and Lagarias \cite{Applegate1995} on the Syracuse coefficient. The interior of every skeleton achieves the minimal possible coefficient because exactly one halving step follows each $3n+1$ operation.

\begin{corollary}[Row preservation]
The diagonal map preserves the row index: if $n$ is at position $(a, b)$ with $a + b = k$, then the next odd is at position $(a-1, b+1)$ with $(a-1) + (b+1) = k$.
\end{corollary}

\begin{corollary}[Interior step count]
From position $(a, b)$ with $a \geq 2$, reaching the boundary $a = 1$ takes exactly $a - 1$ diagonal steps. The exit position is $(1, b + a - 1)$.
\end{corollary}

\subsection{Boundary Behavior}

At the skeleton boundary where $a = 1$, the diagonal dynamics terminate and the trajectory exits from the current skeleton and enters directly into a different skeleton.

\begin{proposition}[Boundary transition]\label{prop:boundary}
For $n = 2\lambda \cdot 3^b - 1$ at the boundary $a = 1$:
\[
v_2(3n + 1) = 1 + v_2(\lambda \cdot 3^{b+1} - 1).
\]
The trajectory exits skeleton $\Skel_\lambda$ and enters a skeleton $\Skel_{\lambda'}$ where $\lambda'$ depends on the factorization of $\lambda \cdot 3^{b+1} - 1$.
\end{proposition}

\begin{proof}
Compute $3n + 1 = 6\lambda \cdot 3^b - 2 = 2(\lambda \cdot 3^{b+1} - 1)$, so $v_2(3n+1) = 1 + v_2(\lambda \cdot 3^{b+1} - 1)$.

The next odd in the trajectory is the odd part of $\lambda \cdot 3^{b+1} - 1$. Writing this odd part as $\lambda' \cdot 2^{a'} \cdot 3^{b'} - 1$ determines the new skeleton $\Skel_{\lambda'}$.
\end{proof}

The interior dynamics are completely deterministic, requiring no factorization. The boundary is where number-theoretic complexity enters. Determining which skeleton a trajectory enters after exiting at boundary position $b$ requires only the odd part of $\lambda \cdot 3^{b+1} - 1$, which is elementary for a single transition. However, predicting the full sequence of boundary transitions is not addressed here by the coordinate system. It is at these boundary transitions, where $a=1$, that the number-theoretic difficulty of the Collatz conjecture resides. 

\subsection{Connection to Growth Rates}

The diagonal flow has a natural growth rate.

\begin{proposition}[Growth rate]
Along diagonal flow from $(a, b)$ to $(a-1, b+1)$, the value increases by a factor approaching $3/2$. Precisely, if $n = \lambda \cdot 2^a \cdot 3^b - 1$ and $n' = \lambda \cdot 2^{a-1} \cdot 3^{b+1} - 1$, then
\[
\frac{n' + 1}{n + 1} = \frac{3}{2}.
\]
\end{proposition}

\begin{proof}
We have $(n' + 1)/(n + 1) = (\lambda \cdot 2^{a-1} \cdot 3^{b+1})/(\lambda \cdot 2^a \cdot 3^b) = 3/2$.
\end{proof}

The ratio $3/2$ connects to the extensive literature on the $\beta$-transformation $x \mapsto (3/2)x \mod 1$ studied by Flatto~\cite{Flatto1992}. The expansion rate emerges naturally from the skeleton parameterization.

\section{Zero-Prime Rows}\label{sec:zero-prime}

The principal skeleton $\Skel_1$ exhibits a striking pattern in the distribution of primes.

\begin{theorem}[Zero-Prime Rows]\label{thm:zero-prime}
For $k \equiv 2 \pmod 4$ with $k \geq 6$, row $k$ of the skeleton $\Skel_1$ contains no primes.
\end{theorem}

\begin{proof}
The elements of row $k$ in $\Skel_1$ are $\{2^a \cdot 3^b - 1 : a + b = k, a \geq 1\}$.

\textbf{Case 1: $a$ and $b$ both even.} Since $k \equiv 2 \pmod 4$, if $a$ is even then $b = k - a$ has the same parity as $k$, which is even. Write $a = 2m$ and $b = 2n$. Then
\[
2^a \cdot 3^b - 1 = (2^m \cdot 3^n)^2 - 1 = (2^m \cdot 3^n - 1)(2^m \cdot 3^n + 1).
\]
For $k \geq 2$, we have $m + n = k/2 \geq 1$, so $2^m \cdot 3^n \geq 2$. Both factors exceed 1, so the element is composite.

\textbf{Case 2: $a$ and $b$ both odd.} Since $k \equiv 2 \pmod 4$ and $a + b = k$, if $a$ is odd then $b$ is also odd. Moreover, $a \equiv b \pmod 4$: two odd numbers summing to $k \equiv 2 \pmod 4$ must both be congruent to 1 or both congruent to 3 modulo 4.

The multiplicative orders modulo 5 are $\mathrm{ord}_5(2) = 4$ and $\mathrm{ord}_5(3) = 4$:
\begin{align*}
2^1 &\equiv 2, & 2^2 &\equiv 4, & 2^3 &\equiv 3, & 2^4 &\equiv 1 \pmod 5, \\
3^1 &\equiv 3, & 3^2 &\equiv 4, & 3^3 &\equiv 2, & 3^4 &\equiv 1 \pmod 5.
\end{align*}

If $a \equiv b \equiv 1 \pmod 4$: Then $2^a \equiv 2 \pmod 5$ and $3^b \equiv 3 \pmod 5$, so $2^a \cdot 3^b \equiv 6 \equiv 1 \pmod 5$.

If $a \equiv b \equiv 3 \pmod 4$: Then $2^a \equiv 3 \pmod 5$ and $3^b \equiv 2 \pmod 5$, so $2^a \cdot 3^b \equiv 6 \equiv 1 \pmod 5$.

In both cases, $2^a \cdot 3^b - 1 \equiv 0 \pmod 5$. For $k \geq 6$, all elements exceed 5, so they are composite.
\end{proof}

The zero-prime theorem raises a natural question: why is the residue class $k \equiv 2 \pmod 4$ special? The following theorem provides a complete answer.

\begin{theorem}[Complete Covering Theorem]\label{thm:complete-covering}
Let $k \geq 2$ and consider positions $(a, b)$ in row $k$ of $\Skel_1$ (i.e., $a + b = k$, $a \geq 1$). The two obstructions partition positions as follows:
\begin{enumerate}
    \item \textbf{Difference of squares (DoS):} Applies when $a$ and $b$ are both even.
    \item \textbf{Divisibility by 5:} Applies when $a$ and $b$ are both odd with $a \equiv b \pmod 4$.
\end{enumerate}
The row class $k \bmod 4$ determines which positions appear:
\begin{enumerate}
    \item[(a)] If $k \equiv 2 \pmod 4$: All positions have $a \equiv b \pmod 2$. Even-even positions satisfy DoS. Odd-odd positions satisfy $a \equiv b \pmod 4$ (since two odd numbers summing to $2 \bmod 4$ must be congruent mod 4). \textbf{Every position is obstructed.}
    \item[(b)] If $k$ is odd: All positions have $a \not\equiv b \pmod 2$. Neither obstruction applies. \textbf{All $k$ positions are prime-admissible.}
    \item[(c)] If $k \equiv 0 \pmod 4$: Even-even positions satisfy DoS, but odd-odd positions have $a \not\equiv b \pmod 4$ (since two odd numbers summing to $0 \bmod 4$ must be incongruent mod 4). \textbf{Exactly $k/2$ positions are prime-admissible.}
\end{enumerate}
\end{theorem}

\begin{proof}
For (a): Let $k \equiv 2 \pmod 4$. If $a, b$ both even, DoS applies. If $a, b$ both odd, write $a = 2m+1$ and $b = 2n+1$. Then $a + b = 2(m+n+1) = k \equiv 2 \pmod 4$, so $m + n \equiv 0 \pmod 2$. Since $m + n$ is even, $m$ and $n$ have the same parity, so $m \equiv n \pmod 2$, which gives $a = 2m+1 \equiv 2n+1 = b \pmod 4$.

For (b): If $k$ is odd, then $a + b$ odd implies $a \not\equiv b \pmod 2$. Neither obstruction applies (DoS requires both even and div-by-5 requires both odd).

For (c): Let $k \equiv 0 \pmod 4$. For even-even positions, DoS applies. For odd-odd positions, write $a = 2m+1$, $b = 2n+1$. Then $m + n + 1 = k/2 \equiv 0 \pmod 2$, so $m + n$ is odd, hence $m \not\equiv n \pmod 2$, giving $a \not\equiv b \pmod 4$. Divisibility by 5 fails.

For the count, since $k \equiv 0 \pmod 4$ implies $k$ is even, the positions $a = 1, 2, \ldots, k$ alternate between $(a, b)$ both odd and $(a, b)$ both even. Row $k$ therefore has $k/2$ positions with both coordinates even (obstructed by DoS) and $k/2$ positions with both coordinates odd (prime-admissible).
\end{proof}

\begin{corollary}[Unique Complete Covering]\label{cor:unique-covering}
Among all residue classes modulo 4, the class $k \equiv 2 \pmod 4$ with $k \geq 6$ is the unique class where every position in $\Skel_1$ is algebraically obstructed.
\end{corollary}

\begin{definition}[Prime-admissible]
A position $(a, b)$ in row $k$ of $\Skel_1$ is \emph{prime-admissible} if it is covered by neither the difference-of-squares obstruction nor the divisibility-by-5 obstruction.
\end{definition}

\begin{remark}[The exception at $k = 2$]
The constraint $k \geq 6$ is necessary because row $k = 2$ contains $\{3, 5\}$, both prime. The Complete Covering Theorem confirms that algebraic obstructions do apply at $k = 2$: position $(2,0)$ factors as $4 - 1 = 1 \times 3$ via difference of squares, and position $(1,1)$ gives $2 \cdot 3 - 1 = 5 \equiv 0 \pmod 5$. However, the obstructed values are the small primes 3 and 5 themselves. For $k \geq 6$, all obstructed values exceed 5 and are therefore composite.
\end{remark}

\begin{remark}[Prime-admissible position counts]
The Complete Covering Theorem yields exact counts of prime-admissible positions per row in $\Skel_1$:
\begin{center}
\begin{tabular}{ll}
\toprule
Row class & Prime-admissible positions in row $k$ \\
\midrule
$k$ odd & All $k$ positions \\
$k \equiv 0 \pmod 4$ & Only $k/2$ positions \\
$k \equiv 2 \pmod 4$, $k \geq 6$ & No positions \\
\bottomrule
\end{tabular}
\end{center}
\end{remark}

\begin{remark}[Generalization to other skeletons]\label{rem:generalization}
The complete covering phenomenon extends beyond $\Skel_1$, but the covered row class depends on $\lambda$. Two conditions govern the obstructions for $n = \lambda \cdot 2^a \cdot 3^b - 1$:
\begin{enumerate}
    \item \textbf{Difference of squares:} Requires $\lambda$ to be a perfect square and $a, b$ both even.
    \item \textbf{Divisibility by 5:} Requires $\lambda \cdot 2^a \cdot 3^b \equiv 1 \pmod 5$, which depends on $\lambda \bmod 5$.
\end{enumerate}
For $\lambda \equiv 1 \pmod 5$ with $\lambda$ a perfect square (including $\lambda = 1$), divisibility by 5 covers positions with $a \equiv b \pmod 4$, and rows $k \equiv 2 \pmod 4$ are completely obstructed. For $\lambda \equiv 4 \pmod 5$ with $\lambda$ a perfect square (e.g., $\lambda = 49$), the divisibility pattern shifts to cover $a \equiv b + 2 \pmod 4$, making rows $k \equiv 0 \pmod 4$ completely obstructed instead. To verify the shift: for $\lambda \equiv 4 \pmod 5$ with $a, b$ both odd, we have $\lambda \cdot 2^a \cdot 3^b \equiv 4 \cdot 2^a \cdot 3^b \pmod 5$. Using $\text{ord}_5(2) = \text{ord}_5(3) = 4$, the cases $a \equiv 1, b \equiv 3 \pmod 4$ give $4 \cdot 2 \cdot 2 = 16 \equiv 1$ and $a \equiv 3, b \equiv 1$ give $4 \cdot 3 \cdot 3 = 36 \equiv 1$. So the obstruction covers positions with $a \not\equiv b \pmod 4$, which characterizes rows $k \equiv 0 \pmod 4$. For non-square $\lambda$ or $\lambda \equiv 0 \pmod 5$, no row class admits complete coverage. Computational verification confirms that such skeletons contain primes in all row classes (see Figure~\ref{fig:zero-prime}).
\end{remark}

\begin{figure}[ht]
\centering
\includegraphics[width=1.0\textwidth]{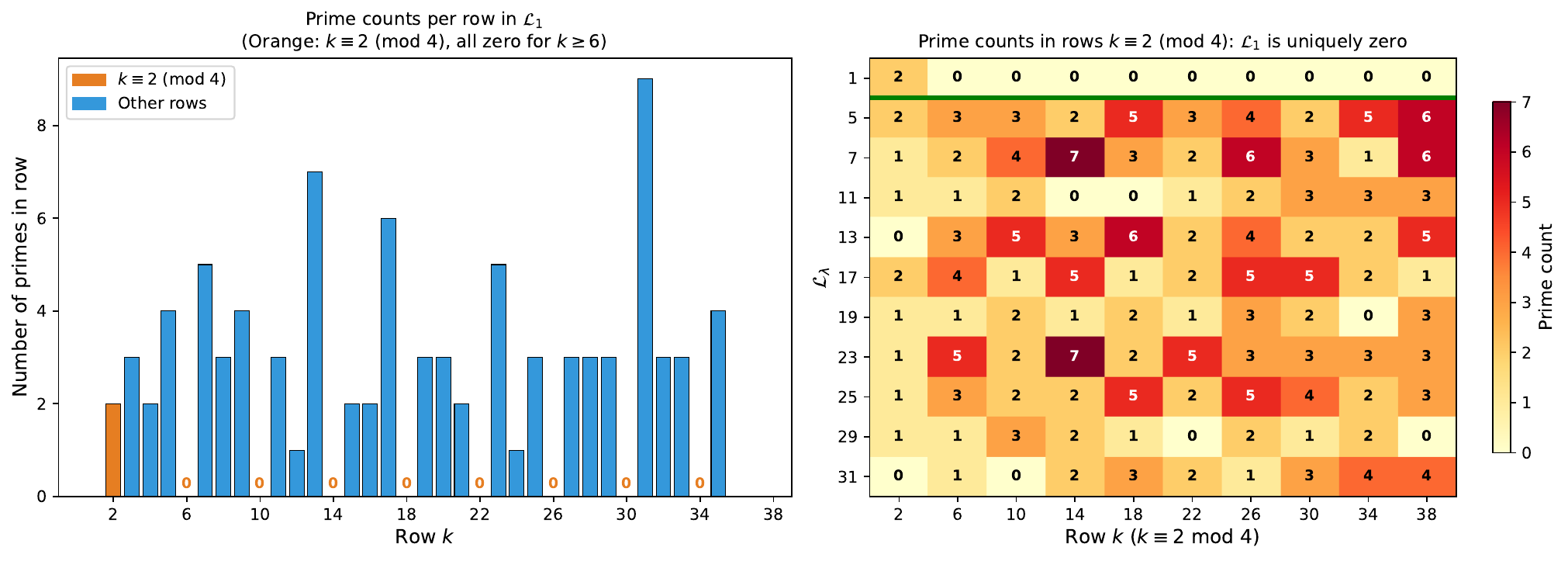}
\caption{Prime counts in rows $k \equiv 2 \pmod 4$. \textbf{Left:} In skeleton $\Skel_1$, rows with $k \equiv 2 \pmod 4$ contain zero primes for $k \geq 6$; the row $k = 2$ is an exception, containing the primes 3 and 5. \textbf{Right:} Across 11 skeletons ($\lambda \in \{1, 5, 7, 11, 13, 17, 19, 23, 25, 29, 31\}$) and 10 rows (110 pairs total), only $\Skel_1$ exhibits this zero-prime property; all other skeletons contain primes in these rows.}
\label{fig:zero-prime}
\end{figure}

The proof uses only elementary divisibility. The contribution of the crown framework is the organization that makes the pattern visible. Without the geometric skeleton structure, zero-primes may not be as obvious.

\section{Connections to Classical Sequences}\label{sec:connections}

The skeleton $\Skel_1$ organizes several classical number-theoretic sequences.

\subsection{Mersenne, Thabit, and Pierpont Numbers}

\begin{proposition}[Mersenne and Thabit numbers]\label{prop:mersenne-thabit}
In $\Skel_1$, the element at position $(a, b)$ is $2^a \cdot 3^b - 1$. Thus position $(k, 0)$ contains $2^k - 1$, the $k$-th Mersenne number, and position $(k-1, 1)$ contains $3 \cdot 2^{k-1} - 1$, a Thabit number.
\end{proposition}

\begin{figure}[H]
    \centering
    \includegraphics[width=\textwidth]{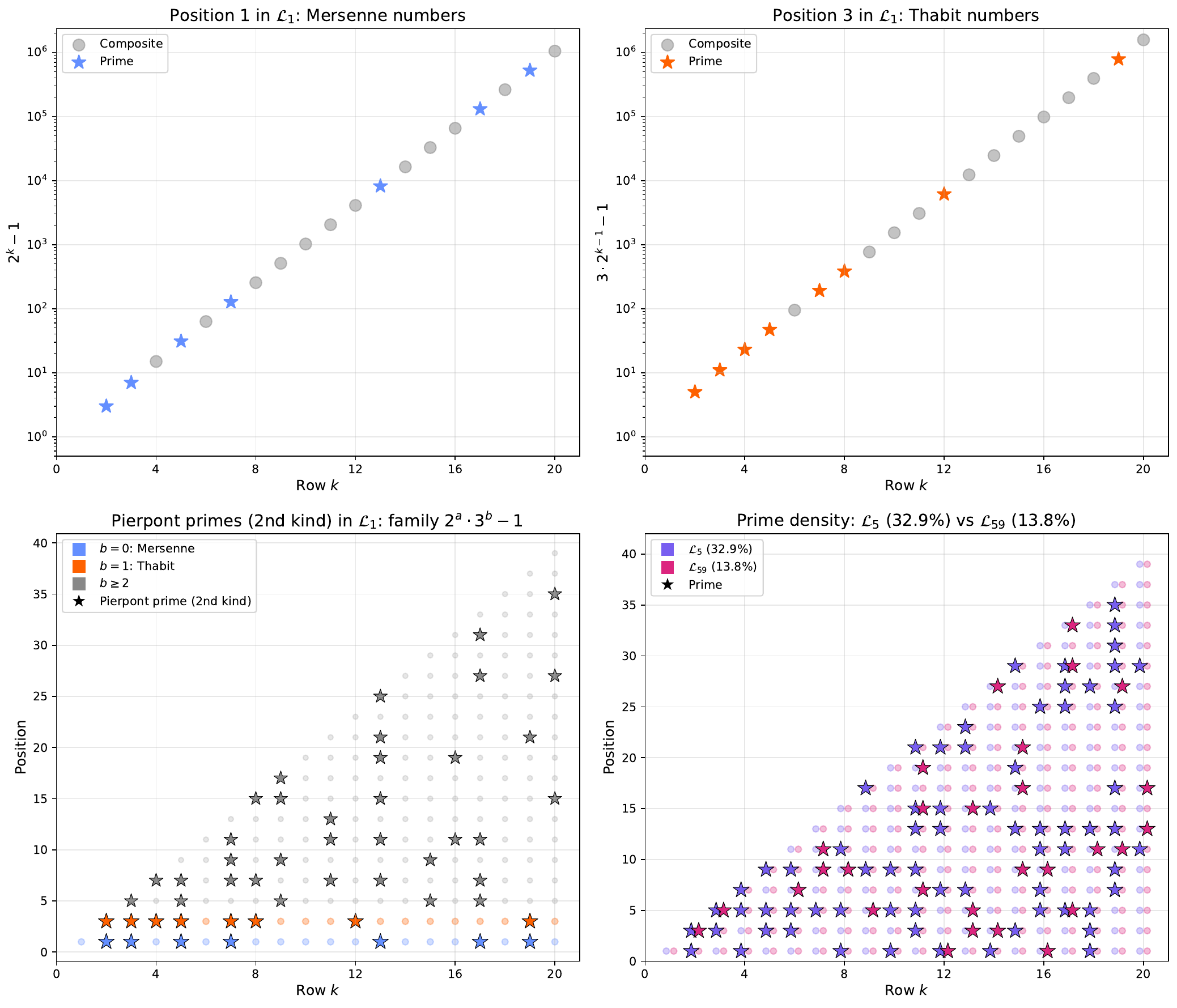}
    \caption{%
        \textbf{Top left:} Mersenne numbers $2^k - 1$ (position 1, i.e., $b = 0$, in $\Skel_1$), with primes starred.
        \textbf{Top right:} Thabit numbers $3 \cdot 2^{k-1} - 1$ (position 3, i.e., $b = 1$, in $\Skel_1$).
        \textbf{Bottom left:} All odd positions in $\Skel_1$, showing the Pierpont family $2^a \cdot 3^b - 1$. Position $2b + 1$ corresponds to exponent $b$. Blue indicates $b = 0$ (Mersenne), orange indicates $b = 1$ (Thabit), and gray indicates $b \geq 2$.
        \textbf{Bottom right:} Prime density comparison between $\Skel_5$ (purple, 32.9\%) and $\Skel_{59}$ (magenta, 13.8\%). Percentages indicate the fraction of odd-position elements that are prime, across rows $k = 1$ to $20$.
    }
    \label{fig:pierpont}
\end{figure}

The Mersenne numbers and Thabit numbers (see the Online Encyclopedia of Integer Sequences (OEIS)~\cite{OEIS} entries A000225 and A055010) are two classical sequences that appear as columns of skeleton $\Skel_1$ (see Figure~\ref{fig:pierpont}). Primes in $\Skel_1$ are exactly the Pierpont primes of the second kind (OEIS A005105). Thabit primes, in particular, appear in classical constructions of amicable number pairs \cite{Guy2004}.

Prior work on 3-smooth numbers studies their distribution and representation properties \cite{Blecksmith1998}. The crown triangle framework organizes the family $\{2^a \cdot 3^b - 1 : a \geq 1, b \geq 0\}$ by weight $a + b$. Our weight-based organization re-framed as a coordinate system indexing Collatz chains (this paper) appears not to have been studied previously.

\subsection{OEIS Connections}

The crown sequence has exact relationships to several OEIS sequences.

\begin{center}
\begin{tabular}{llll}
\toprule
OEIS & Definition & Relation to Crowns \\
\midrule
A007494 & $n \equiv 0, 2 \pmod 3$ & Crowns $= 4 \cdot \text{A007494}$ \\
A016777 & $n \equiv 1 \pmod 3$ & Complement of A007494 \\
A008594 & $n \equiv 0 \pmod{12}$ & Crowns $\equiv 0 \pmod{12}$ \\
A017557 & $n \equiv 8 \pmod{12}$ & Crowns $\equiv 8 \pmod{12}$ \\
A005105 & Pierpont primes (2nd kind) & Primes in $\Skel_1$ \\
\bottomrule
\end{tabular}
\end{center}

The sequences A007494 and A016777 partition $\Znn$. A007494 contains integers congruent to 0 or 2 modulo 3, while A016777 contains those congruent to 1 modulo 3. Crowns are precisely the multiples of 4 not in the range of the Collatz odd step $n \mapsto 3n + 1$.

\section{Computational Observations and Asymptotic Results}\label{sec:observations}

We record observations verified computationally but not yet proved. These suggest structure that may merit further investigation.

\subsection{Crown Visitation Asymmetry}
\begin{observation}[Visitation asymmetry]\label{obs:asymmetry}
For Collatz orbits starting from $n = 1, 2, \ldots, 10^7$, let each orbit contribute its set of distinct crown triangles visited. Averaged over the orbits for these starting values, 18.2 unique crown triangles are visited per orbit. Of these, 93.6\% are from the $c \equiv 8 \pmod{12}$ family and 6.4\% are from the $c \equiv 0 \pmod{12}$ family. This is despite the two crown families having equal density among the integers 0 to $10^7$.
\end{observation}
The asymmetry arises because nearly all orbits pass through only one crown with $c \equiv 0 \pmod{12}$, namely crown~0 at termination, while visiting many different crowns with $c \equiv 8 \pmod{12}$ en route.

\subsection{The Case of $n=27$}\label{subsec:27}
The orbit of $n=27$ is famously anomalous~\cite{Lagarias2010, lagarias1985}. Starting from $n=27$, the Collatz trajectory reaches a maximum of 9232 before eventually descending to 1 after 111 steps. This behavior has puzzled investigators since the problem's introduction.

\newpage
In the crown framework, $n = 27$ has factorization $28 = 2^2 \cdot 7$, placing it at position $(2, 0)$ in row $k = 2$ of $\Skel_7$. The next odd integer in its orbit, 41 (with $42 = 2 \cdot 3 \cdot 7$), is also in $\Skel_7$, at the boundary position $(1, 1)$. From this boundary the orbit exits $\Skel_7$ and passes through a sequence of skeletons with much larger $\lambda$ values. We extend the skeleton assignment from odd integers to all integers in an orbit by the convention that each step inherits the $\lambda$ of the next crown reached. Under this convention, the orbit of 27 visits 9 distinct skeletons with $\lambda \in \{335, 1823, 2309, 245, 47, 41, 11, 5, 1\}$, reaching a maximum $\lambda = 2309$ before eventually descending to $\Skel_1$.

\begin{figure}[H]
    \centering
    \includegraphics[width=\textwidth]{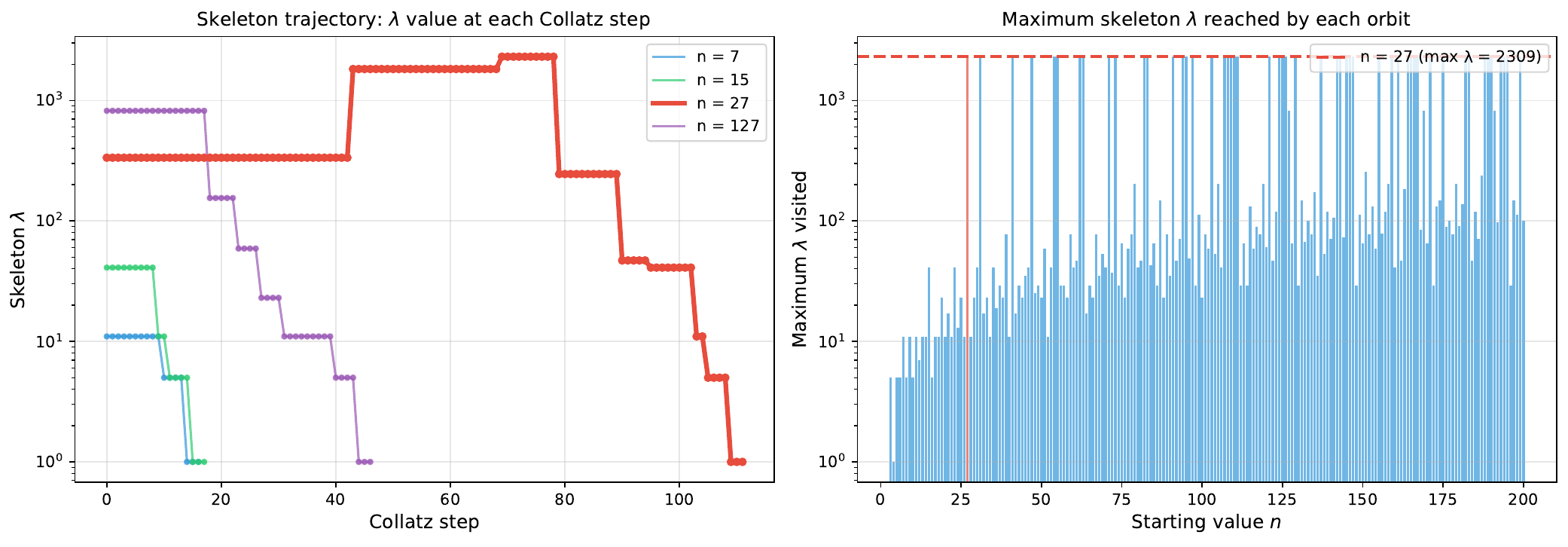}
    \caption{%
        \textbf{Left:} Skeleton $\lambda$ value at each Collatz step for selected starting values. The orbit of $n=27$ (red) reaches much higher $\lambda$ values compared with other starting $n$ values.
        \textbf{Right:} Maximum $\lambda$ visited by orbits starting from $n = 3, 4, \ldots, 200$. The value $n = 27$ (red bar) is a local anomaly, reaching $\lambda = 2309$.
    }
    \label{fig:orbit-27}
\end{figure}

This observation does not explain why the orbit of $n=27$ behaves anomalously compared to other values of $n$ (see Figure~\ref{fig:orbit-27}). However, it contextualizes the anomaly geometrically in terms of crown triangle structure and skeleton transitions.

\subsection{Chain Length Distribution}

The edge formulas from Lemma~\ref{lem:edges} explain the distribution of chain lengths.

\begin{proposition}[Chain count]\label{prop:chain-count}
Let $R_k(N)$ denote the number of Collatz chains of length $2k+1$ containing at least one element in $\{0, 1, \ldots, N\}$. For $k \geq 2$,
\[
R_k(N) = \frac{N}{6 \cdot 2^{k-1}} + O(1).
\]
\end{proposition}

\begin{proof}
Row $k$ of $\Tri_c$ intersects $[0, N]$ when its minimum element $2^k\lambda - 1 \leq N$, giving $\lambda \leq (N+1)/2^k$. Since $\lambda$ ranges over integers coprime to 6, the count of such $\lambda$ is $(N+1)/(3 \cdot 2^k) + O(1)$, giving the result.
\end{proof}

The ratio $R_k(N)/R_{k+1}(N) \to 2$ as $N \to \infty$, so the number of chains roughly halves with each increase in $k$. Computational verification for $N$ up to $10^7$ and $k \in \{2, 3, 4, 5\}$ confirms this.

\subsection{Accidentally Prime-Free Rows}

The Complete Covering Theorem (Theorem~\ref{thm:complete-covering}) characterizes \emph{algebraically} prime-free rows, meaning those where every position is covered by the difference-of-squares or divisibility-by-5 obstruction. However, computational verification reveals additional prime-free rows that escape algebraic explanation. 

\begin{observation}[Accidentally prime-free rows]\label{obs:accidental}
Among rows with $2 \leq k \leq 300$ in $\Skel_1$, some rows contain prime-admissible positions but no actual primes.
\begin{itemize}
    \item For $k \equiv 0 \pmod 4$, the prime-free rows are $84, 100, 116, 156, 176, 184, 188, 200, 252, 284, 300$, comprising 14.7\% of such rows.
    \item Among odd rows $k \ge 3$ the prime-free rows are $149, 165, 261$, comprising $3/149 \approx 2.0\%$ of such rows.
    \item $k=1$ is excluded as its single element is 1, which is neither prime nor composite.
\end{itemize}
We call these rows \emph{accidentally prime-free} to distinguish them from the \emph{algebraically prime-free} rows with $k \equiv 2 \pmod 4$ and $k \geq 6$.
\end{observation}

The existence of accidentally prime-free rows demonstrates that algebraic admissibility does not guarantee prime existence. This is consistent with general heuristics. For large integers, primality is rare, and a row with $k$ elements of size roughly $2^k$ has expected prime count approximately $k/(k \ln 2) = 1/\ln 2 \approx 1.44$ by the prime number theorem. As $k$ increases, this expectation remains roughly constant while the variance increases, so prime-free rows should occur with positive probability.

The distinction between algebraically and accidentally prime-free rows is significant. The former admit complete characterization (Corollary~\ref{cor:unique-covering}), while the latter appear governed by probabilistic considerations beyond the two identified obstructions.

\subsection{Computational Methods}\label{sec:methods}

We describe here the computational infrastructure underlying the observations and figures in this section. All computations were carried out in Python 3. Primality testing uses a deterministic Miller-Rabin implementation with the witness set $\{2, 3, 5, 7, 11, 13, 17, 19, 23, 29, 31, 37\}$, which gives certified primality for all integers below $2^{64}$ \cite{SorensonWebster2017}. The larger elements arising in Theorem~\ref{thm:zero-prime} and Observation~\ref{obs:accidental}, which exceed this bound, are classified using \texttt{sympy.isprime}, as described below. The prime counts in Figure~\ref{fig:zero-prime} and Figure~\ref{fig:pierpont} were computed with \texttt{sympy.isprime}, which returns deterministic results at the sizes appearing in those figures.

Chain enumeration for Figure~\ref{fig:chain-lengths} and Proposition~\ref{prop:chain-count} uses the explicit row formula $\Tri_c(k, 0) = 2(2^k \lambda - 1)$ from Lemma~\ref{lem:edges}. For each crown $c$ with minimum element in $[0, N]$, the row is generated by iterating the Collatz map from $\Tri_c(k, 0)$ for $2k$ steps. The distinction between chains fully contained in $[0, N]$ and chains with at least one element outside this interval is tracked by comparing the maximum element $\Tri_c(k, 2k) = 2(3^k \lambda - 1)$ against $N$. Enumeration runs to $N = 10^7$.

The crown visitation statistics in Observation~\ref{obs:asymmetry} were computed by generating full Collatz orbits for each starting value $n = 1, 2, \ldots, 10^7$ and recording the set of distinct crown triangles touched along each orbit. A crown triangle is identified by checking the residue $n \bmod 12$ at each step. The reported averages and percentages are aggregated over all $10^7$ starting orbits.

The skeleton trajectory in Section~\ref{subsec:27} and Figure~\ref{fig:orbit-27} uses the extended skeleton assignment described earlier, where each integer in an orbit inherits the $\lambda$ value of the next crown reached. This is implemented by walking forward from each orbit step until a value $c \equiv 0$ or $8 \pmod{12}$ is encountered, then computing $\lambda = (c+2)/2$. This extension is used only for visualization of skeleton trajectories.

The prime-admissible position counts reported after Corollary~\ref{cor:unique-covering} were verified directly by enumerating all positions $(a, b)$ with $a + b = k$, $a \geq 1$, for $k = 1, \ldots, 42$ and checking each against the difference-of-squares and divisibility-by-5 obstructions from Theorem~\ref{thm:complete-covering}. The observed counts match the theoretical predictions exactly.

The accidentally prime-free rows reported in Observation~\ref{obs:accidental} were identified by testing each element of each row $k \leq 300$ in $\Skel_1$ for primality. The largest element tested is $2 \cdot 3^{299} - 1$, with approximately 143 decimal digits. Elements below $2^{64}$ are certified prime or composite by the deterministic Miller-Rabin test, while the larger elements, which exceed that bound, are classified using \texttt{sympy}'s \texttt{isprime} function, which combines strong Miller-Rabin and Lucas tests and has no known counterexamples. For each row with $k \not\equiv 2 \pmod 4$, we record the row as accidentally prime-free if no element is prime.

Code for all computations and figures is available at \url{https://github.com/rhoposit/collatz-coordinate-system}.

\section{Discussion and Open Problems}\label{sec:discussion}

We have introduced a coordinate system for nonnegative odd integers based on 3-smooth factorization. The crown triangle partition organizes all nonnegative integers into arrays whose rows are Collatz chains. Each crown triangle can be expressed as a skeleton consisting of only the odd elements. The coordinate system reframes Collatz dynamics as diagonal flow within skeletons, making the local structure explicit and isolating the number-theoretic difficulty at boundary transitions.

As a concrete application, we proved that rows $k \equiv 2 \pmod 4$ with $k \geq 6$ in the principal skeleton $\Skel_1$ contain no primes. The Complete Covering Theorem shows this is the unique residue class admitting complete algebraic obstruction, providing a full characterization of which row classes can be proven prime-free by elementary divisibility arguments.

The coordinate system is new. The algebraic identities underlying the proofs are elementary. The contribution of this paper is the framework that makes the geometric structure visible and enables these characterization results.

We summarize here several open problems extending our work:

\begin{enumerate}

\item \textbf{Row-exit dynamics.} When a trajectory exits skeleton $\Skel_\lambda$ at the boundary, which skeleton does it enter? Each individual transition is elementary to compute, but the map sending one skeleton to the next has no known closed form. A complete characterization of this transition map would give a description of Collatz dynamics in skeleton coordinates.

\item \textbf{Why is $p = 3$ special?} The $3n+1$ map admits a 3-smooth coordinate system because multiplying by 3 and adding 1 transforms $\lambda \cdot 2^a \cdot 3^b - 1$ into $2(\lambda \cdot 2^{a-1} \cdot 3^{b+1} - 1)$. This preserves $\lambda$ while shifting $(a, b)$ to $(a-1, b+1)$, giving a diagonal flow within each skeleton for $a \ge 2$. The operative identity is $3-1 = 2$. For the general $pn+1$ map, write $n = \lambda\cdot 2^a\cdot p^b - 1$. Then $pn+1 = \lambda\cdot 2^a\cdot p^{b+1}-(p-1)$, and the subtracted term $p-1$ is a power of $2$ only when $p-1 = 2^c$, and $p=3$ is the unique case with $c=1$, which preserves the weight $a+b$. Whether alternative coordinate systems have a role in characterizing $pn+1$ dynamics is unknown.

\item \textbf{$\lambda$-recurrence.} Computationally, every trajectory we tested (starting from $n = 1, 2, \ldots, 10^7$) eventually reaches some $n$ with $\lambda(n) = 1$. This is implied by the Collatz conjecture, since the cycle $1 \to 4 \to 2 \to 1$ lies in $\Tri_0$ with $\lambda = 1$. But proving that every trajectory reaches $\lambda = 1$ (call this $\lambda$-recurrence) might be easier than proving full convergence. Can the number of steps to reach $\lambda = 1$ be bounded?

\item \textbf{Statistical structure.} Tao \cite{Tao2022} showed that almost all orbits attain almost bounded values. Can the skeleton framework provide a geometric interpretation of this result?

\item \textbf{Parity vectors and lattice paths.} Knight \cite{Knight2026} characterizes Collatz cycles using parity vectors, which record sequences of odd and even steps. In skeleton coordinates, odd steps correspond to diagonal moves $(a, b) \to (a-1, b+1)$. The relationship between parity vector combinatorics and skeleton geometry has not been investigated.

\item \textbf{Accidentally prime-free rows.} Observation~\ref{obs:accidental} identifies rows in $\Skel_1$ that contain prime-admissible positions but no primes. Is the set of such rows finite or infinite? What is the asymptotic density within each residue class? For $k \equiv 0 \pmod 4$, approximately 15\% of rows up to $k = 300$ are accidentally prime-free; does this proportion persist, vanish, or converge to a positive limit?

\end{enumerate}

The results of this paper stand independent of the Collatz conjecture. The partition theorem organizes all nonnegative integers by a 3-smooth coordinate system. The Complete Covering Theorem characterizes prime-free rows in the Pierpont family $\{2^a \cdot 3^b - 1\}$, a result in the classical study of primes in shifted smooth number families. The coordinate system unifies Mersenne numbers, Thabit numbers, and their generalizations into a single geometric framework. Whether this organization yields further insight into Collatz dynamics remains open, but the structure it reveals may be of interest more widely.
\section*{Acknowledgments}

The author would like to express sincere thanks to Marc Chamberland for insightful discussions and encouragement during the early stages of this work. The author is also grateful to Michaeel Kazi, Mykel Kochenderfer, Michael Brandstein, Ramona Comanescu, Catherine Lai, Alireza Goudarzi, and Poppy Welch for their support, and to Henna Chhohan and William Toner for reviewing early drafts of this manuscript. The author is especially grateful to the University of Southampton Pure Mathematics Group for their helpful discussions, critiques, and curiosities during the final stages of this work, and in particular to Ian Leary for spotting a minor error.

\section*{Funding Statement}
No funding was obtained for the reported work.

\section*{Disclosure Statement}
The authors report there are no competing interests to declare.

\section*{Declaration of Generative AI Use}
No artificial intelligence tools were used for the underlying conceptual ideation, core hypothesis formulation, or foundational discovery of this work. The generative AI tool Claude (Anthropic, Version: Claude Opus 4.8) was used substantially during this work to assist in developing and drafting proof arguments and structures, which the author critically checked, revised, and verified; to assist with writing computational scripts; and to format text. Every mathematical proof was thoroughly scrutinized by the author, and any logical or computational errors introduced by the tool were manually identified and corrected. In strict accordance with publisher policies, all AI-assisted content underwent rigorous human revision, testing, and validation. The author compiled and verified all code pipelines locally to ensure data integrity, and assumes sole academic responsibility for the mathematical accuracy, validity, and integrity of the final manuscript and its replication materials.

\section*{Data Availability Statement}
The code supporting the findings of this study is openly available at \url{https://github.com/rhoposit/collatz-coordinate-system}.

\bibliographystyle{tfnlm}

\begin{thebibliography}{99}

\bibitem{Applegate1995}
D.~Applegate and J.~C. Lagarias.
\newblock Density bounds for the $3x+1$ problem. {I}. {T}ree-search method.
\newblock {\em Mathematics of Computation}, 64(209):411--426, 1995.

\bibitem{Barina2021}
D.~Bařina.
\newblock Convergence verification of the {C}ollatz problem.
\newblock {\em The Journal of Supercomputing}, 77(3):2681--2688, 2021.

\bibitem{BatemanHorn1962}
P.~T.~Bateman and R.~A.~Horn. A heuristic asymptotic formula concerning the distribution of prime numbers. \textit{Mathematics of Computation}, 16(79):363--367, 1962.

\bibitem{Blecksmith1998}
R.~Blecksmith, M.~McCallum, and J.~L.~Selfridge.
\newblock 3-smooth representations of integers.
\newblock {\em American Mathematical Monthly}, 105(6):529--543, 1998.

\bibitem{Flatto1992}
L.~Flatto. {Z}-numbers and $\beta$-transformations. \textit{Symbolic Dynamics and Its Applications} (P.~Walters, ed.), vol.~135 of \textit{Contemporary Mathematics}, pp.~181--201, American Mathematical Society, 1992.

\bibitem{Guy2004}
R.~K.~Guy. \textit{Unsolved Problems in Number Theory}, 3rd~ed., Problem Books in Mathematics, vol.~1. Springer, New York, 2004.

\bibitem{Kannan2016}
T.~Kannan and C.~Ganesa~Moorthy. Collatz Conjecture for Modulo an Integer. \textit{International Journal of Mathematics and its Applications}, 4(3-A):41--61, 2016.

\bibitem{Knight2026}
K.~Knight.
\newblock Collatz high cycles do not exist.
\newblock {\em Discrete Mathematics}, 349:114812, 2026.

\bibitem{lagarias1985}
J.~C.~Lagarias. The $3x+1$ problem and its generalizations. \textit{The American Mathematical Monthly}, 92(1):3--23, 1985.

\bibitem{Lagarias2010}
J.~C.~Lagarias. \textit{The Ultimate Challenge: The $3x+1$ Problem}. American Mathematical Society, 2010.

\bibitem{LagariasWeiss1992}
J. C. Lagarias and A. Weiss. The $3x+1$ problem: Two stochastic models. \emph{Annals of Applied Probability}, 2(1):229--261, 1992.

\bibitem{Mahler1968}
K.~Mahler. An unsolved problem on the powers of 3/2. \textit{Journal of the Australian Mathematical Society}, 8(2):313--321, 1968.

\bibitem{OEIS}
{OEIS Foundation Inc.}, ``The On-Line Encyclopedia of Integer Sequences,'' \url{https://oeis.org}, 2024.

\bibitem{Ren2023}
W.~Ren. Collatz dynamics is partitioned by residue class regularly. \textit{Research in Mathematics}, 10(1):2269657, 2023.

\bibitem{SorensonWebster2017}
J.~Sorenson and J.~Webster.
\newblock Strong pseudoprimes to twelve prime bases.
\newblock {\em Mathematics of Computation}, 86(304):985--1003, 2017.

\bibitem{Tao2022}
T.~Tao.
\newblock Almost all orbits of the {C}ollatz map attain almost bounded values.
\newblock {\em Forum of Mathematics, Pi}, 10:e12, 2022.

\bibitem{terras1976stopping}
R.~Terras. A stopping time problem on the positive integers. \textit{Acta Arithmetica}, 30:241--252, 1976.

\bibitem{Wirsching1998}
G.~Wirsching. \textit{The Dynamical System Generated by the $3n+1$ Function}, Lecture Notes in Mathematics, No.~1681. Springer-Verlag: Berlin, 1998.

\end{thebibliography}

\end{document}